\documentclass[12pt, twoside]{article}
\usepackage{amsmath,amsthm,amssymb}
\usepackage{times}
\usepackage{enumerate}
\usepackage{graphicx}
\usepackage{caption}
\usepackage{subcaption}
\usepackage{verbatim}

\theoremstyle{definition}
\newtheorem{thm}{Theorem}[section]

\newtheorem{defin}[thm]{Definition}
\newtheorem{rem}[thm]{Remark}

\newtheorem{pro}[thm]{Proposition}
\newtheorem{prb}[thm]{Problem}

\numberwithin{equation}{section}

\newcommand{\subjclass}[1]{\bigskip\noindent\emph{2010 Mathematics Subject Classification:}\enspace#1}
\newcommand{\keywords}[1]{\noindent\emph{Keywords:}\enspace#1}
\newcommand{\hf}{u}
\newcommand{\spon}{H_s}	%spontaneous curvature

\newcommand{\mc}{H}
\newcommand{\gc}{K}
\newcommand{\gbr}{\kappa_G}

\newcommand{\mbr}{\kappa}
\newcommand{\st}{\sigma}
\newcommand{\lt}{b}

\def\Xint#1{\mathchoice
	{\XXint\displaystyle\textstyle{#1}}%
	{\XXint\textstyle\scriptstyle{#1}}%
	{\XXint\scriptstyle\scriptscriptstyle{#1}}%
	{\XXint\scriptscriptstyle\scriptscriptstyle{#1}}%
	\!\int}
\def\XXint#1#2#3{{\setbox0=\hbox{$#1{#2#3}{\int}$ }
		\vcenter{\hbox{$#2#3$ }}\kern-.6\wd0}}

\def\dashint{\Xint-}

\usepackage{color}
\definecolor{newblue}{RGB}{90,150,250}
\definecolor{newblue2}{cmyk}{1,0.6,0,0.06}
\definecolor{grey}{gray}{0.5}
\usepackage[pdfborder={0 0 0},
  colorlinks=true,
  citecolor=newblue,
  linkcolor=newblue2,
  urlcolor=newblue]{hyperref}
\usepackage[nameinlink]{cleveref}

\usepackage{xcolor}
\definecolor{Gr}{HTML}{006400}

\usepackage[normalem]{ulem}
\usepackage[nameinlink]{cleveref}

\usepackage{autonum}
\begin{document}

%%%%% To ease editing, add:

\baselineskip=17pt

%%%%%%%%%%%%%%%%

\title{On the  sharp interface limit of a phase  field model for near-spherical two phase biomembranes}

\author{\textsc{Charles M. Elliott}\\
\small Mathematics Institute,
University of Warwick, Coventry, CV4 7AL, United Kingdom\\
\small E-mail: C.M.Elliott@warwick.ac.uk\\\\
\textsc{Luke Hatcher}\\
\small Mathematics Institute,
University of Warwick, Coventry, CV4 7AL, United Kingdom\\
\small E-mail: L.Hatcher@warwick.ac.uk\\\\
\textsc{Bj{\"o}rn Stinner}\\
\small Mathematics Institute and Centre for Scientific Computing,\\
\small University of Warwick, Coventry, CV4 7AL, United Kingdom\\
\small E-mail: Bjorn.Stinner@warwick.ac.uk}

\date{}

\maketitle

%%%%%%%%

\begin{abstract}
We consider  sharp interface asymptotics for a phase field model of two phase near spherical biomembranes  involving a coupling between the local mean curvature and the local composition proposed by the first and second authors. The model is motivated by lipid raft formation.  
We introduce a reduced diffuse interface energy depending only on the membrane composition 
and derive the $\Gamma-$limit.
We demonstrate that the {Euler-Lagrange} equations for the limiting functional and the sharp interface energy coincide.
Finally, we consider a system of gradient flow equations with conserved Allen-Cahn dynamics for the phase field model. 
Performing a formal asymptotic analysis
we obtain a system of gradient flow equations for the sharp interface energy coupling geodesic curvature flow for the phase interface  to a fourth order PDE free  boundary problem for the surface deformation.

%A template for submissions to IFB. The subject classification and keywords are optional. For the subject classification, use
%the 2010 Mathematics Subject Classification available at www.ams.org/msc.


\end{abstract}

%%%%%%%%%%%%%%%%%%%%%%%%%%%%%%%%%%%%%%%%%%%%%%%%%%%%%%%%%%%%%%%%%%%%%%%%%%%%%%%%%%%%%
\section{Introduction}
Biological membranes are lipid bilayers which separate a cell's interior from it's exterior and often contain embedded molecules such as proteins. Biomembranes also exhibit fluid-like properties which enables the lateral transport of these molecules and can lead to the formation of intramembrane domains \cite{bassereau2018physics}.  In this paper we consider a mathematical model in which domains are one phase of a two phase biomembrane.
Since the length scales of a biomembrane are much larger than its width, biomembranes are typically modelled by hypersurfaces and the introduction of surface energy functionals.

 In \cite{elliott_hatcher_2020} the first and second authors  considered surfaces  
$\Gamma_\rho$, of the form
\begin{align}
\label{eqn-graph}
\Gamma_\rho=\{x+\rho u(x)\nu(x):x\in \Gamma\}
\end{align}
where $\Gamma$ is as sphere of radius $R$.  A surface of this type is a graph over the  base surface $\Gamma$ with unit normal $\nu$ and   described by a height function $u:\Gamma\to\mathbb{R}$ with small positive constant $\rho$.  In \cite{elliott_hatcher_2020}  the following energy was derived

\begin{equation}
\label{eqn-diffuse-energy}
\mathcal{E}_{DI}(\hf,\phi) = \int_{\Gamma}  \left( e_m(\phi, u,\Delta_\Gamma u)+e_{DI}(\phi,\nabla_\Gamma \phi)\right )
\end{equation}
where an  approximate membrane elastic energy, $e_m(\phi,u,\Delta_\Gamma u)$,  and a diffuse interface energy, $e_{DI}(\phi,\nabla_\Gamma\phi),$ are given   by
\begin{align}
e_m(u,\Delta_\Gamma u,\phi):= &
\frac{\kappa}{2}\left(( \Delta_{\Gamma} u+\frac{2u}{R^2}+H_s(\phi))^2 - \frac{u}{R^2}(\Delta_\Gamma u+\frac{2u}{R^2})\right)+ 
\frac{\sigma }{2}u \left(\Delta_\Gamma u+\frac{2u}{R^2}\right)  \\ \quad 
 e_{DI}(\phi,\nabla_\Gamma\phi):= & b\left(\frac{\epsilon}{2} |\nabla_\Gamma\phi|^2 + \frac{1}{\epsilon} W(\phi)\right) .
\end{align}

By applying a perturbation method introduced in \cite{elliott2016small} (see also \cite{elliott2019small}), it was shown that (\ref{eqn-diffuse-energy})
 approximates the Canham-Helfrich energy functional \cite{canham1970minimum,helfrich1973elastic} 
\begin{equation}
\label{eqn-helfrich+tension+assumptions}
\mathcal{F}_{DI}(\Gamma,\phi):=\int_{\Gamma}\left(\frac{1}{2}\mbr(\mc-\spon(\phi))^2+\st\right)+  \lt\int_\Gamma \left (\frac{\epsilon}{2}|\nabla_\Gamma\phi|^2+\frac{1}{\epsilon}W(\phi)\right).
\end{equation}

 The first term in  (\ref{eqn-helfrich+tension+assumptions}) is a Canham -Helfrich surface energy. Here $\mc$ is the mean curvature of $\Gamma$. The parameter $\mbr>0$ is a bending rigidity and $\st\geq 0$ is the surface tension.  The membrane composition is given by  the order parameter  $\phi:\Gamma\to\mathbb{R}$.  $\spon(\phi)\equiv \Lambda \phi$ is a composition dependent spontaneous curvature. The coefficient $\Lambda>0$ couples the local order parameter $\phi$, to the local membrane curvature. 
 Note that in the  surface energy  we have omitted $\gbr\gc$ where $\gc$ is the Gauss curvature  and  $\gbr$ is a bending rigidity constant. This is valid when considering hypersurfaces of constant genus.

Domains are distinguished by the  use of  $W(\cdot)$,  a double well potential defined by $W(\phi):=\frac{1}{4}(\phi^2-1)^2$. ${b}$ is a diffuse interface energy coefficient associated with the phase boundary  and $\epsilon>0$ is a small parameter commensurate with the width of a diffuse interface separating the two phases. As $\epsilon\to 0$ then $\phi$ is forced to the roots of $W(\cdot)$ given by $\phi=\pm 1$ with these values  corresponding to the two  phases. The phase boundary is then the level set $\phi=0$ on $\Gamma$.

In this paper we wish to relate the diffuse interface energy \eqref{eqn-diffuse-energy} as $\epsilon \rightarrow 0$ to the  sharp interface energy
\begin{equation}
\label{eqn-sharp-interface-energy}
\mathcal E_{SI}(u,\gamma)=\int_{\Gamma }e_m(\chi_\gamma, u,\Delta_\Gamma u) + \int_{\gamma} \hat{b}
\end{equation} 
where the sphere $\Gamma$ is decomposed into subsets $\Gamma^{(1)}$ and $\Gamma^{(2)}$ both with common boundary $\gamma$ and $\chi_\gamma=1$ in $\Gamma^{(2)}$ and $\chi_\gamma=-1$ in $\Gamma^{(1)}$. The line energy coefficient $\hat{b}$ is scaled with the diffuse interface  energy  coefficient $b$ and  double well potential $W$ so  $\hat \lt=c_W\lt$, where $c_W:=\int_{-1}^{1}\sqrt{2W(s)}\:{\rm d}s=\frac{2\sqrt{2}}{3}$.

Related to \eqref{eqn-helfrich+tension+assumptions} is the sharp interface elastic energy first introduced by J{\"u}licher and Lipowsky \cite{julicher1993domain,julicher1996shape} and given by,
\begin{equation}
\label{eqn-helfrich}
\mathcal{F}_{SI}(\Gamma,\gamma):=\int_{\Gamma^{(1)}\cup\Gamma^{(2)}}\left(\frac{1}{2}\mbr(\mc-\spon^i)^2+\st \right) +\int_\gamma \hat{b}.
\end{equation}
for hypersurfaces $\Gamma=\Gamma^{(1)}\cup\gamma\cup\Gamma^{(2)}$. For the axisymmetic case a minimisation problem has been addressed \cite{choksi2013global} and numerical simulations explored \cite{garcke2020structure}. In the non-axisymmetric case very little has been rigorously proven although Brazda et al. has recently dealt with the minimisation problem in the weaker setting of oriented curvature varifolds \cite{brazda2019existence}. Analogous to the diffuse interface approach the same perturbation method could be applied to \eqref{eqn-helfrich}, to obtain \eqref{eqn-sharp-interface-energy}.

\
We relate the diffuse interface approach to the sharp interface approach in two ways. 
\begin{enumerate}
\item
First, by calculating the Euler-Lagrange equations we express the height function $u$ in terms of the functions $\phi$ and $\chi_\gamma$ for the diffuse and sharp interface energies respectively. By substituting these into our energies we eliminate $u$ and obtain a reduced diffuse interface energy $\widetilde{\mathcal{E}}_{DI}(\phi)$, and a reduced sharp interface energy $\widetilde{\mathcal{E}}_{SI}(\gamma)$ in terms of only $\phi$ and $\chi_\gamma$. We prove that a suitable minimisation problem of the energy \eqref{eqn-diffuse-energy} coincides with the associated minimisation problem of the reduced diffuse energy $\widetilde{\mathcal{E}}_{DI}(\phi)$. A similar result is shown for the sharp interface energies. Returning to the reduced diffuse interface energy $\widetilde{\mathcal{E}}_{DI}(\phi)$, we calculate its $\Gamma-$limit by showing it can be written as the  as the Modica-Mortola functional plus a continuous perturbation. Finally, we relate the two approaches by showing that minimisers of the $\Gamma-$limit coincide with minimisers of the reduced sharp interface energy $\widetilde{\mathcal{E}}_{SI}(\gamma)$.

\item
Secondly, since gradient flow methods are often used to numerically investigate critical points, we consider a gradient flow of \eqref{eqn-diffuse-energy} with conserved Allen-Cahn dynamics which was considered in \cite{elliott_hatcher_2020}. Again, we apply a reduction method  which enables us to write the fourth order equation for the height function as a second order equation, but at the cost that the reduced order equation contains a non-local operator. A formal asymptotic analysis of the limit $\epsilon \to 0$ for more general surfaces was performed in \cite{elliott2010surface} but only for the equilibrium equations. Here we perform it for the time-dependent problem and show how the non-local term which arises from this reduction method can be dealt with. We show that the resulting free boundary problem coincides with a corresponding conserved $L^2$-gradient flow for the sharp interface energy \eqref{eqn-sharp-interface-energy}.

\end{enumerate}

\subsection{Background}
Examples of domains  are lipid rafts. These are small (10-200nm), heterogeneous domains which compartmentalise cellular processes and are enriched with various molecules such as cholestorol and sphingolipids, and which can form larger platforms through protein-protein and protein-lipid interactions \cite{pike2006rafts}. They were first introduced by Simons in \cite{simons1997functional}, but have received large academic interest since, for example see \cite{dimova2019giant,goni2019rafts,schmid2017physical,sezgin2017mystery} and their references.
For technical reasons, direct microscopic detection of lipid rafts has not been possible. However, domain formation  has been observed on large artificial membranes  for which the curvature of the membrane plays an important role \cite{baumgart2003imaging, rinaldin2020geometric}. Numerical simulations,  \cite{elliott_hatcher_2020}, of the model considered in this paper display domain formation similar to those occurring in experiments, \cite{baumgart2003imaging}.

The coupling of the elastic energy to a composition field was first considered by Leibler in \cite{leibler1986curvature}. More recently it  has been considered using computational and formal asymptotic perspectives \cite{elliott2010surface,elliott2010modeling,elliott2013computation},  applying a bifurcation analysis \cite{healey2017symmetry},  addressing non-equilibrium properties such as dissipation effects \cite{tozzi2019out}, and  calculating the $\Gamma-$limit in the axisymmetric case\cite{helmers2011snapping,helmers2013kinks,helmers2015convergence}. Note that this last example differs from our work since the biomembranes are only assumed to be $C^0$, which allows for kinks across the interface between domains.

Our work here extends that of Ren and Wei in \cite{ren2004soliton}, who determine the $\Gamma-$limit in the one dimensional  approximately planar case, although they limit themselves to only considering a one-dimensional problem. This differs from our work which is for two-dimensional approximately spherical surfaces. 
Our work also differs from that of Fonseca et al.  in \cite{fonseca2016domain}, who focus on surface tension effects and consider $\Gamma-$convergence of an approximately planar surface but for a different parameter regime.

\subsection{Outline}
The outline for the rest of this paper is as follows.  In Section \ref{Sec:Notation} we briefly cover some notation and preliminaries needed for this paper. In Section \ref{Sec:Diffuse} we derive the reduced diffuse interface energy $\widetilde{\mathcal{E}}_{DI}(\phi)$ and calculate it's $\Gamma-$limit. We then state regularity results for minimisers of the $\Gamma-$limit. In Section \ref{sec:varSI(main)} we calculate the Euler-Lagrange equations for the sharp interface energy, and use these to obtain the reduced sharp interface energy $\widetilde{\mathcal{E}}_{SI}(\gamma)$ which coincides with the $\Gamma-$limit obtained in Section \ref{Sec:Diffuse} for suitably regular solutions. In Section \ref{Sec-Formal} we perform the formal asymptotic analysis for the diffuse interface gradient flow equations and show the resulting free boundary problems coincides with the corresponding sharp interface gradient flow equations. Finally, in Section \ref{Sec-Conclusion} we finish with some concluding remarks.

%%%%%%%%%%%%%%%%%%%%%%%%%%%%%%%%%%%%%%%%%%%%%%%%%%%%%%%%%%%%%%%%%%%%%%%%%%%%%%%%%%%%%

\section{Notation and preliminaries}
\label{Sec:Notation}

{
	Here we outline some important calculus results for stationary and evolving surfaces. 
}
For a thorough treatment of the material covered here we refer the reader to \cite{dziuk2013finite}.

Although throughout the paper $\Gamma$ is the sphere with radius $R$, here, in this section,  we present some notation for general   oriented two-dimensional hypersurfaces $\Gamma$ that are smooth with smooth boundary (unless stated otherwise). 
Suppose $x\in\Gamma$ and $U$ is an open subset containing $x$. Then given a function $u\in C^1(U)$ we define the surface gradient $\nabla_\Gamma u(x)$ of $u$ at $x$ by
\begin{align}
\nabla_\Gamma u(x)=\nabla u(x)-(\nabla u(x)\cdot \nu(x))\nu(x),
\end{align}
where $\nu$ is a smooth unit normal field to $\Gamma$. Note that this derivative depends on the values of $u$ on $\Gamma$ only. Denoting it's components by
\begin{align}
\nabla_\Gamma u=(\underline{D}_1u,\underline{D}_2u,\underline{D}_3u).
\end{align}
we can also define the Laplace-Beltrami operator of $u$ at $x$ by
\begin{align}
\Delta_\Gamma u(x)=\sum_{i=1}^{3}\underline{D}_i\underline{D}_i u(x)
\end{align}
provided that $u \in C^2(U)$. We define the Lebesgue space $L^p(\Gamma)$ for $p\in[1,\infty)$ to be the space of functions which are measurable with respect to the surface measure $\mathrm{d}\Gamma$ and have finite norm
\begin{align}
\|u\|_{L^p(\Gamma)}=\left(\int_{\Gamma}^{}|u|^p \:\mathrm{d}\Gamma\right)^\frac{1}{p}.
\end{align}
We say a function $u\in L^1(\Gamma)$ has the weak derivative $v_i=\underline{D}_iu$, if for every function $\phi\in C^1_0(\Gamma)$ we have the relation
\begin{align} \label{eqn-weakderivative}
\int_{\Gamma}^{}u\underline{D}_i\phi \:\mathrm{d}\Gamma=-\int_{\Gamma}^{}\phi v_i \:\mathrm{d}\Gamma+\int_{\Gamma}^{}u\phi H\nu_i\:\mathrm{d}\Gamma,
\end{align}
where $H$ is the mean curvature of $\Gamma$.

We define the Sobolev space $W^{1,p}(\Gamma)$ and Hilbert spaces $H^1(\Gamma)$ and $H^2(\Gamma)$ by
\begin{align}
W^{1,p}(\Gamma):&=\left\{f\in L^p(\Gamma):f\text{ has weak derivatives }\underline{D}_if\in L^p(\Gamma), i\in\{1,2,3\}\right\},
\\
H^1(\Gamma):&=\left\{f\in L^2(\Gamma):f\text{ has weak derivatives }\underline{D}_if\in L^2(\Gamma), i\in\{1,2,3\}\right\},
\\
H^2(\Gamma):&=\left\{f\in H^1(\Gamma):f\text{ has weak derivatives }\underline{D}_i\underline{D}_jf\in L^2(\Gamma), i,j\in\{1,2,3\}\right\}.
\end{align}
In additional we say a function $f\in L^1(\Gamma)$ has bounded variation and write $f\in BV(\Gamma)$ if
\begin{align}
|Df|(\Gamma):=\sup_{\eta\in C^1_c(\Gamma;\mathbb{R}^3)}\left\{\int_{\Gamma}f\nabla_\Gamma\cdot \eta\:{\rm d}\Gamma :|\eta|\leq 1\right\}<\infty
\end{align}
Here, $|Df|(\Gamma)$ is known as the total variation of $f$. We will use the notation $BV(\Gamma;\{-1,1\})$ to denote a function of bounded variation on $\Gamma$ which only takes values $\pm 1$.

Integration by parts on bounded $C^2-$ hypersurfaces reads (Theorem 2.10 and 2.14, \cite{dziuk2013finite}):
	\begin{align}
	\int_{\Gamma}\nabla_\Gamma\cdot f\:\mathrm{d}\Gamma&=\int_{\Gamma}f\cdot H\nu \mathrm{d}\Gamma+\int_{\partial\Gamma}f\cdot\nu_{\partial\Gamma}\:\:\mathrm{d}(\partial\Gamma) \label{eqn-ByParts1} \\
	\int_{\Gamma}\nabla_\Gamma \eta\cdot\nabla_\Gamma v \:\:{\rm d}\Gamma&=-\int_{\Gamma}\eta\Delta_\Gamma v\:\: {\rm d}\Gamma+\int_{\partial\Gamma}\eta\nabla_\Gamma v\cdot\nu_{\partial\Gamma}\:\:\mathrm{d}(\partial\Gamma). \label{eqn-ByParts2}
	\end{align}
	for $f\in W^{1,1}(\Gamma,\mathbb{R}^3)$, $\eta\in H^1(\Gamma)$, $v\in H^2(\Gamma)$ and where $\nu_{\partial\Gamma}$ denotes the conormal to $\gamma$ .

%%%%%%%%%%%%%%%%%%%%%%%%%%%%%%%%%%%%%%%%%%%%%%%%%%%%%%%%%%%%%%%%%%%%%%%%%%%%%%%%%%%%%%%%%%%%%%%%%%%%%%%%%%%%%%%%%%%%%%%%%%%%%%%%%%%%%%%%%%%%%%%%%%%%%%%%%%%%%

\section{Diffuse interface energy  minimisation}
\label{Sec:Diffuse}
\subsection
	{Diffuse interface minimisers}
	
	Let  $\Gamma$  be the sphere of radius $R$ and consider the diffuse interface energy $\mathcal{E}_{DI}(u,\phi)$ as given in \eqref{eqn-diffuse-energy} for $u\in H^2(\Gamma)$, $\phi\in H^1(\Gamma)$ and equal to $+\infty$ 
	{
		if $u \in L^2(\Gamma) \backslash H^2(\Gamma)$ or $\phi\in L^1(\Gamma)\backslash H^1(\Gamma)$.
	}
	
	For $(u,\phi)\in H^2(\Gamma)\times H^1(\Gamma)$, consider the constraints  ($\dashint_\Gamma:=\frac{1}{|\Gamma|}\int_\Gamma$)
	\begin{align}
	\label{eqn-constraint-DI}
	&\dashint_\Gamma \phi\:{\rm d}\Gamma=\alpha,&
	&\int_{\Gamma} u \:{\rm d}\Gamma =0,&
	&\int_{\Gamma} \nu_i 
	u\:{\rm d}\Gamma =0\qquad\text{ for }i\in\{1,2,3\}.
	\end{align}
	Here, $\alpha\in[-1,1]$ and $\nu_i$ for $i\in\{1,2,3\}$ are the three components of the normal $\nu(x)=x/|x|$ at $x\in \Gamma$. 
	We define the space $\mathcal{K}_{DI}$ as follows
	\begin{align}
	\mathcal{K}_{DI}:=\left\{(u,\phi)\in H^2(\Gamma)\times H^1(\Gamma):(u,\phi)\text{ satisfy }\eqref{eqn-constraint-DI}\right\}
	\end{align}
	The diffuse interface minimisation problem is:
		\begin{prb}
			\label{prob-diffuse}
			Find $(u^*,\phi^*)\in \mathcal{K}_{DI}$ such that
						\begin{align}
						\mathcal{E}_{DI}(u^*,\phi^*)=\inf_{(u,\phi)\in\mathcal{K}_{DI}}\mathcal{E}_{DI}(u,\phi).
						\end{align}
		\end{prb}

	The first condition of \eqref{eqn-constraint-DI} corresponds to a conservation of mass constraint  on the order parameter $\phi$, the second condition is a volume constraint and relates to impermeability of the membrane, and the third condition is a nullspace constraint related to a translation invariance property of the membrane energy. The functional is coercive over this set,  see \cite{elliott_hatcher_2020}, so there exist minimisers. Proceeding as in \cite{elliott_hatcher_2020}, we may calculate the first variation of \eqref{eqn-diffuse-energy} to be
	
	\begin{multline}
	\label{eqn-1st-var-DI}
	\left<\mathcal{E}^\prime_{DI}(u,\phi),(\zeta,\eta)\right> = \int_{\Gamma} \bigg(\kappa\Delta_\Gamma u\Delta_\Gamma \zeta +\left(\sigma-\frac{2\kappa}{R^2}\right)\nabla_\Gamma u\cdot\nabla_\Gamma \zeta 
	-\frac{2\sigma}{R^2} u\zeta +\kappa\Lambda\phi\Delta_\Gamma \zeta+\frac{2\kappa\Lambda}{R^2}\phi\zeta \\ + \kappa\Lambda\Delta_\Gamma u \eta+\frac{2\kappa\Lambda}{R^2}u  \eta  +\kappa\Lambda^2\phi\eta +{b\epsilon}\nabla_\Gamma\phi\cdot\nabla_\Gamma \eta +\frac{b}{\epsilon}W^\prime(\phi)\eta\Bigg{)} \:{\rm d}\Gamma. 
	\end{multline}
and the variations of the constraints \eqref{eqn-constraint-DI}, solve
	\begin{align}
	\label{eqn-Var-constraint-DI}
	&\int_\Gamma \eta\:{\rm d}\Gamma=0,&
	&\int_{\Gamma} \zeta \:{\rm d}\Gamma =0,&
	&\int_{\Gamma} \nu_i 
	\zeta\:{\rm d}\Gamma =0\qquad\text{ for }i\in\{1,2,3\},	\end{align}
for all $\zeta\in H^2(\Gamma)$ and for all $\eta\in H^1(\Gamma)$. For further details see \cite{elliott_hatcher_2020}, Section 4.1.  The  Lagrange multipliers corresponding to the constraints  for the constrained optimisation problem, Problem \ref{prob-diffuse}, may be determined easily and we obtain the following form of Euler-Lagrange equations:
	\begin{align}
	\label{eqn-diffuse-EL2}
	\frac{\hat b}{\epsilon}\left(W^\prime(\phi^*)-\dashint_\Gamma W^\prime(\phi^*)\:{\rm d}\Gamma\right)-\hat b\epsilon\Delta_\Gamma \phi^*+\kappa\Lambda\Delta_\Gamma u^*+\frac{2\kappa\Lambda}{R^2}u^* +\kappa\Lambda^2(\phi^*-\alpha) =0,&
	\\
	\label{eqn-diffuse-EL}
	\left(\Delta_\Gamma+\frac{2}{R^2}\right)\left(\kappa\Delta_\Gamma u^*-\sigma u^*+\kappa\Lambda(\phi^*-\alpha)\right)=0.&
	\end{align}

\subsection{Reduced diffuse interface energy}
\label{Reduced diffuse interface energy}
The Euler-Lagrange equation \eqref{eqn-diffuse-EL} motivates  seeking a reduced order PDE. 
We begin by noting that if 
\begin{align}
-\Delta_\Gamma z -\frac{2}{R^2}z=0,
\end{align}
then $z$ is an eigenfunction of $-\Delta_\Gamma$ with eigenvalue $\frac{2}{R^2}$. Such eigenfunctions $z$ belong to the space $\text{span}\{\nu_1,\nu_2,\nu_3\}$ (see \cite{elliott2016small}).  It is convenient to work with the $L^2(\Gamma)$ orthogonal complement of  $\text{span}\{1,\nu_1,\nu_2,\nu_3\}$ and we   set
 $$ S:=\text{span}\{1,\nu_1,\nu_2,\nu_3\}^\perp.$$
  Note that if $\eta \in S\cup H^2(\Gamma)$ then, since  $\nu_i$ are eigenfunctions of $-\Delta_\Gamma$, a short calculation shows that  $\Delta_\Gamma \eta \in S$.
 Also  it is convenient to define  an operator $\mathcal G:S\rightarrow H^2(\Gamma)\cap S$ where for each $\eta\in S$ 
 $\mathcal{G}(\eta)$ is the  unique function in $H^2(\Gamma) \cap \mathcal{S}$ satisfying 
\begin{align}
\label{eqn-green}
\left({\sigma}-\kappa\Delta_\Gamma\right)\mathcal{G}(\eta)=\kappa\Lambda\eta.
\end{align}

Let $(u^*,\phi^*)$ be a diffuse interface energy minimiser.  It follows from \eqref{eqn-diffuse-EL}  that 
$-\kappa\Delta_\Gamma u^*+\sigma u^*-\kappa\Lambda( (\phi^*-\alpha)\in \text{span}\{\nu_1,\nu_2,\nu_3\}$
 so we may write\begin{align}
\label{eqn-remark3.2-1}
-\kappa\Delta_\Gamma u^*+\sigma u^*=\kappa\Lambda( (\phi^*-\alpha)- &\beta_{DI} )& & \text{on } \Gamma.
\end{align}
 where  $\beta_{DI} \in\text{span}\{\nu_1,\nu_2,\nu_3\}$.
 Denoting by  $\mathbf P:L^2(\Gamma)\rightarrow S$   the orthogonal projection onto $S$, we find after applying it to \eqref{eqn-remark3.2-1}

\begin{align}
\label{eqn:red-ord-2}
\left(\sigma -\kappa\Delta_\Gamma \right)u^*=\kappa\Lambda\mathbf{P}(\phi^*), 
\end{align}
so
\begin{align}
\label{eqn-Defhf}
u^*=\mathcal G(\mathbf P(\phi^*))
\end{align}

This motivates  defining  the \emph{reduced diffuse interface energy}, $\widetilde{\mathcal{E}}_{DI}(\phi)$ 
\begin{align}\label{eqn-refucedenergy}
\begin{split}
\widetilde{\mathcal{E}}_{DI}(\phi):=&\mathcal{E}_{DI}(\mathcal{G}(\mathbf{P}(\phi)),\phi)
\\=&
\int_{\Gamma}  \bigg(\frac{\kappa}{2}\left( \Delta_{\Gamma} \mathcal{G}(\mathbf{P}(\phi))+\frac{2\mathcal{G}(\mathbf{P}(\phi))}{R^2}+\Lambda \phi\right)^2 
\\
&-\left(\frac{\kappa}{R^2}+ \frac{\sigma }{2}\right)\mathcal{G}(\mathbf{P}(\phi)) \left(\Delta_{\Gamma} +\frac{2}{R^2}\right)\mathcal{G}(\mathbf{P}(\phi))
+\frac{b\epsilon}{2} |\nabla_{\Gamma}\phi|^2 + \frac{b}{\epsilon} W(\phi)\bigg)\:{\rm d}\Gamma
\\
=&\int_{\Gamma}\bigg(\frac{\kappa}{2}\left(\Delta_\Gamma-\frac{\sigma}{\kappa}\right)\mathcal{G}(\mathbf{P}(\phi))\left(\Delta_\Gamma+\frac{2}{R^2}\right)\mathcal{G}(\mathbf{P}(\phi))+\frac{\kappa\Lambda^2\phi^2}{2}
\\
&\qquad+{\kappa\Lambda}\phi\left(\Delta_\Gamma+\frac{2}{R^2}\right)\mathcal{G}(\mathbf{P}(\phi))+\frac{b\epsilon}{2}|\nabla_{\Gamma}\phi|^2+\frac{b}{\epsilon}W(\phi)\bigg)\:{\rm d}\Gamma
\end{split}
\end{align}
 and the admissible set
\begin{align}
\widetilde{\mathcal{K}}_{DI}:=&\left\{\phi\in H^1(\Gamma):\dashint_\Gamma \phi\:{\rm d}\Gamma=\alpha\right\}.
\end{align}
Using \eqref{eqn:red-ord-2} and that $\left(\Delta_\Gamma+\frac{2}{R^2}\right)\mathcal{G}(\mathbf{P}(\phi))\in {S}$ we can simplify \eqref{eqn-refucedenergy} to obtain that
\begin{align}
\label{eqn-refucedenergy2}
\widetilde{\mathcal{E}}_{DI}(\phi)=\int_{\Gamma}\left(\frac{\kappa\Lambda}{2}\mathbf{P}(\phi)\left(\Delta_\Gamma+\frac{2}{R^2}\right)\mathcal{G}(\mathbf{P}(\phi))+\frac{b\epsilon}{2}|\nabla_{\Gamma}\phi|^2+
\frac{b}{\epsilon}W(\phi)+\frac{\kappa\Lambda^2\phi^2}{2}\right)\:{\rm d}\Gamma.
\end{align}
We write the constrained minimisation problem for the reduced energy below.
\begin{prb}
	\label{prob-diffuse-reduced}
	Find $\widetilde{\phi^*}\in \widetilde{\mathcal{K}}_{DI}$ such that
	\begin{align}
	\widetilde{\mathcal{E}}_{DI}(\widetilde{\phi^*})=\inf_{\phi \in \widetilde{\mathcal{K}}_{DI}}\widetilde{\mathcal{E}}_{DI}(\phi).
	\end{align}
\end{prb}
We note that a minimiser of Problem \ref{prob-diffuse-reduced} is equivalent to finding a minimiser of Problem \ref{prob-diffuse} since
\begin{align}
\label{eqn:red-ord-3}
\mathcal{E}_{DI}(u^*,\phi^*) \leq \mathcal{E}_{DI}(\mathcal{G}(\mathbf{P}(\widetilde{\phi}^*)),\widetilde{\phi}^*) = \widetilde{\mathcal{E}}_{DI}(\widetilde{\phi}^*) \leq \widetilde{\mathcal{E}}_{DI}(\phi^*) = \mathcal{E}_{DI}(\mathcal{G}(\mathbf{P}(\phi^*)),\phi^*) = \mathcal{E}_{DI}(u^*,\phi^*).
\end{align}

\subsection{$\Gamma-$convergence}
\label{Sec:Gamma-conv}

We will now calculate the $\Gamma-$limit of Problem \ref{prob-diffuse-reduced} as $\epsilon \to 0$. First we  decompose  the energy \eqref{eqn-refucedenergy2} and write
\begin{align}
\widetilde{\mathcal{E}}_{DI}(\phi)=\mathcal{J}_\epsilon(\phi)+\mathcal{K}(\phi),
\end{align}
where $\mathcal{J}_\epsilon(\phi)$ contains the local, $\epsilon$-dependent part of the energy and $\mathcal{K}(\phi)$ is a non-local, $\epsilon$-independent perturbation, given by
\begin{align}
\mathcal{J}_\epsilon(\phi):=&
\begin{cases}
\int_{\Gamma}\frac{b\epsilon}{2}|\nabla_{\Gamma}\phi|^2+\frac{b}{\epsilon}W(\phi)\:{\rm d}\Gamma&\text{for }\phi\in H^1(\Gamma), \\
+\infty&\text{for }\phi\in L^1(\Gamma)\backslash H^1(\Gamma), 
\end{cases} \\
\label{eqn-K-pre}
\mathcal{K}(\phi):=&
\begin{cases}
\int_{\Gamma}\frac{\kappa\Lambda}{2}\mathbf{P}(\phi)\left(\Delta_\Gamma+\frac{2}{R^2}\right)\mathcal{G}(\mathbf{P}(\phi))+\frac{\kappa\Lambda^2\phi^2}{2}\:{\rm d}\Gamma&\text{for }\phi\in L^2(\Gamma), \\
+\infty&\text{for }\phi\in L^1(\Gamma)\backslash L^2(\Gamma). 
\end{cases}
\end{align}

To calculate the $\Gamma-$limit we reformulate \eqref{eqn-K-pre}. Since
	\begin{align}
	\label{eqn-reformulateK}
	\begin{split}
	\left(\Delta_\Gamma+\frac{2}{R^2}\right)\mathcal{G}(\mathbf{P}(\phi))&=\left( \Big{(} \frac{\sigma}{\kappa}+\frac{2}{R^2} \Big{)} - \Big{(} -\Delta_\Gamma+\frac{\sigma}{\kappa} \Big{)} \right) \mathcal{G}(\mathbf{P}(\phi))
	\\
	&= \Big{(} \frac{\sigma}{\kappa}+\frac{2}{R^2} \Big{)} \mathcal{G}(\mathbf{P}(\phi))-\Lambda \mathbf{P}(\phi)
	\end{split}
	\end{align}
	it follows that 
	\begin{align}
	\label{eqn-K}
	\mathcal{K}(\phi)=&
	\begin{cases}
	\int_{\Gamma}\frac{\kappa\Lambda}{2}\mathbf{P}(\phi)\left(\frac{\sigma}{\kappa}+\frac{2}{R^2}\right)\mathcal{G}(\mathbf{P}(\phi))+\frac{\kappa\Lambda^2}{2}(\phi^2-(\mathbf{P}(\phi))^2)\:{\rm d}\Gamma&\text{for }\phi\in L^2(\Gamma), \\
	+\infty&\text{for }\phi\in L^1(\Gamma)\backslash L^2(\Gamma). 
	\end{cases}
	\end{align}
	Calculating the $\Gamma-$limit is then straightforward:
\begin{pro}
	\label{Prop-GammaLimit}
	The $\Gamma-$limit of $\widetilde{\mathcal{E}}_{DI}(\phi)=\mathcal{J}_\epsilon(\phi)+\mathcal{K}(\phi)$ is given by $\widetilde{\mathcal{E}}_0(\phi):=\mathcal{J}_0(\phi)+\mathcal{K}(\phi)$ with 
	\begin{align} \label{eqn-defJ0}
	\mathcal{J}_0(\phi)=\begin{cases}
	\frac{bc_W}{2}|D\phi|(\Gamma)&\text{for }\phi\in BV(\Gamma;\{-1,1\}), \\
	+\infty&\text{for }\phi\in L^1(\Gamma)\backslash BV(\Gamma;\{-1,1\}), 
	\end{cases}
	\end{align}
	where $c_W=\int_{-1}^{1}\sqrt{2W(s)}\:{\rm d}s=\frac{2\sqrt{2}}{3}$.
\end{pro}
\begin{proof}
	It is known (for example see \cite{alberti2000variational,garcke2013curvature,modica1987gradient}) that $\mathcal{J}_\epsilon(\phi)$ $\Gamma-$converges to the functional $\mathcal{J}_0(\phi)$.
	Furthermore by considering \eqref{eqn-K} and using elliptic regularity, it follows that $\mathcal{K}$ is a continuous functional. 
	$\Gamma-$convergence is stable under continuous perturbations \cite{braides2002gamma}, Remark 1.7. Therefore $\mathcal{J}_\epsilon+\mathcal{K}$, $\Gamma-$converges to $\mathcal{J}_0+\mathcal{K}$ as $\epsilon\to 0$.
\end{proof}

\subsection{Minimisers of the $\Gamma-$limit}

{
	We now consider the following minimisation problem:
	\begin{prb}
		\label{Prob:GL}
		Find $\phi^*\in \mathcal{D}:=\{\eta\in BV(\Gamma;\{-1,1\}):\dashint_\Gamma\eta\:{\rm d}\Gamma=\alpha\}$ such that
		\begin{align}
		\widetilde{\mathcal{E}}_0(\phi^*)=\inf_{\phi\in\mathcal{D}}\widetilde{\mathcal{E}}_0(\phi),
		\end{align}
		where $\widetilde{\mathcal{E}}_0(\phi)=\mathcal{J}_0(\phi)+\mathcal{K}(\phi)$.
	\end{prb}
	This minimisation problem, Problem \ref{Prob:GL}, lies in the class of the \emph{liquid drop problem} as formulated in \cite{maggi2012sets}, Section 19. We comment that since in two-dimensions functions of bounded variation continuously embed in $L^2$ (Corollary 3.49, \cite{ambrosio2000functions}), we can interpret $\phi\in BV(\Gamma;\{-1,1\})$ as a function in $L^2(\Gamma)$. Therefore, by elliptic regularity we have that $\mathcal{G}(\mathbf{P}(\phi))\in H^2(\Gamma)$. Since $H^2(\Gamma)\hookrightarrow C^0(\Gamma)$ in two dimensions then the integrand of $\mathcal{K}$ satisfies the property
	\begin{align}
	\frac{\kappa\Lambda}{2}\phi\left(\frac{\sigma}{\kappa}+\frac{2}{R^2}\right)\mathcal{G}(\mathbf{P}(\phi))+\frac{\kappa\Lambda^2}{2}\left(\phi^2-\left(\mathbf{P}(\phi)\right)^2\right)\in L^\infty(\Gamma).
	\end{align}
	Hence, by Theorem 19.5 in \cite{maggi2012sets} a minimiser $\phi^*$ exists to Problem \ref{Prob:GL}. Furthermore, by applying regularity results (see part III of \cite{maggi2012sets}) we can define $\Gamma^{(1)}:=\{\phi^*=-1\}$, $\Gamma^{(2)}=\{\phi^*=+1\}$, and obtain that $\gamma^*$, the common boundary of $\Gamma^{(1)}$ and $\Gamma^{(2)}$,  has H{\"o}lder regularity of $C^{1,\beta}$ for $\beta \in [0,\frac{1}{2})$. In particular see Theorem 21.8 in \cite{maggi2012sets}. 
	
	By introducing the function $\chi_{\gamma^*}:BV(\Gamma)\to\mathbb{R}$, where
	\begin{align}
	\label{eqn-chi-gamma}
	\chi_{\gamma^*}:=\begin{cases}
	-1&\text{on }\Gamma^{(1)},\\
	+1&\text{on }\Gamma^{(2)},
	\end{cases}
	\end{align}
	we can reformulate the $\Gamma-$limit energy $\mathcal{E}_0(\phi^*)$ to be in terms of the curve $\gamma^*$. Therefore, \eqref{eqn-defJ0} can be written as 
	\begin{align}
	\mathcal{J}_0(\gamma^*)=\int_{\gamma^*}\widehat{b}\:{\rm d}\gamma^*,
	\end{align}
	(see Remarks 12.2 and 12.3 in \cite{maggi2012sets}), and we write \eqref{eqn-K-pre} as
	\begin{align}
	\mathcal{K}(\gamma^*)=\int_{\Gamma}\frac{\kappa\Lambda}{2}\chi_{\gamma^*}\left(\Delta_\Gamma+\frac{2}{R^2}\right)\mathcal{G}(\mathbf{P}(\chi_{\gamma^*}))+\frac{\kappa\Lambda^2\chi_{\gamma^*}^2}{2}\:{\rm d}\Gamma.
	\end{align}
	\begin{rem}
		In fact we can prove greater regularity for the curve $\gamma^*$. Using interior regularity we
		obtain that $\mathcal{G}(\mathbf{P}(\phi))\in C^\infty(\Gamma^{(i)}\backslash\gamma^*)$ for $i\in\{1,2\}$ and hence the integrand of $\mathcal{K}$ satisfies
		\begin{align}
		\label{eqn-SI-smooth}
		\frac{\kappa\Lambda}{2}\chi_{\gamma^*}\left(\Delta_\Gamma+\frac{2}{R^2}\right)\mathcal{G}(\mathbf{P}(\chi_{\gamma^*}))+\frac{\kappa\Lambda^2\chi_{\gamma^*}^2}{2}\in C^\infty(\Gamma^{(i)}\backslash\gamma^*).
		\end{align}
		for $i\in\{1,2\}$. Using \eqref{eqn-SI-smooth} and applying a bootstrapping argument it can be shown that $\gamma^*$ is $C^\infty$, see Section 27 in \cite{maggi2012sets}, and also \cite{philippis2015regularity} and its references. Moreover, using boundary regularity we find that 
		\begin{align}
		\mathcal{G}(\mathbf{P}(\chi_{\gamma^*}))\in C^\infty\left(\overline{\Gamma^{(1)}}\right)\cap C^\infty\left(\overline{\Gamma^{(2)}}\right)\cap C^1(\Gamma).
		\end{align}
	\end{rem}

\section{Sharp interface optimisation problem}
\label{sec:varSI(main)}

The objective of this section is to derive the Euler-Lagrange equations of the sharp interface energy functional $\mathcal E_{SI}(\gamma,u)$ defined in \eqref{eqn-sharp-interface-energy}. Using these to eliminate the membrane height, a reduced energy functional is derived. This is then shown to coincide with the $\Gamma$-limit of the reduced diffuse interface energy derived in the previous section. 

\subsection{Minimisation problem}
\label{sec:Admissible two-phase sharp interface surfaces}
We define $\mathcal{K}_{SI}$ to be the set of all pairs $(u,\gamma)$ satisfying:
\begin{itemize}
	\item $u:\Gamma\to\mathbb{R}$ is a height function such that $u\in H^2(\Gamma)$,
	\item {$\Gamma$ is decomposed as }
	$\Gamma=\Gamma^{(1)}\cup\gamma\cup\Gamma^{(2)}$, where $\gamma$ consists of finitely many, $C^1$ closed curves and is the common boundary of
	hypersurfaces $\Gamma^{(1)}$ and $\Gamma^{(2)}$,
\end{itemize}
and such that $(\gamma,u)$ satisfy the constraints,
\begin{align}
\label{eqn-constraint1}
C_1(\gamma)&:=|\Gamma^{(1)}|-|\Gamma^{(2)}|+\alpha|\Gamma|=0, \\
\label{eqn-constraint2}
C_2(u)&:=\int_{\Gamma}u =0,\\
\label{eqn-constraint3}
\mathcal{N}_i(u)&:=\int_{\Gamma}u\nu_i=0,\qquad\text{ for }i\in\{1,2,3\}.
\end{align}
These constraints correspond to \eqref{eqn-constraint-DI} for the diffuse interface approach. 

We will use the notation of an upper index of the form $(\cdot)^{(1)}$ or $(\cdot)^{(2)}$ to indicate the limit of quantities on $\gamma$ approached from either $\Gamma^{(1)}$ or $\Gamma^{(2)}$ and use $[\cdot]^{(2)}_{(1)}=(\cdot)^{(2)}-(\cdot)^{(1)}$ to denote the jump of a quantity across $\gamma$. We define    $\nu_{\Gamma^{(i)}}$ to be the unit   conormal, tangential to $\Gamma^{(i)}$,  normal   to $\gamma$ and pointing out of $\Gamma^{(i)}$.  Since $\Gamma$ is $C^1$ we may introduce $\mu$ so that
\begin{align}
\mu := \nu_{\Gamma^{(1)}}=-\nu_{\Gamma^{(2)}}.
\end{align}
Furthermore, using that $H^2(\Gamma)\hookrightarrow C^0(\Gamma)$, we have that 
\begin{align}
\label{eqn:jump1}
\left[u\right]^{(2)}_{(1)}=&0&&\text{on }\gamma.
\end{align}
In addition, since $u\in H^2(\Gamma)$, then trace values of the first weak derivatives exist on $\gamma$ for the domains $\Gamma^{(1)}$ and $\Gamma^{(2)}$, and these trace values coincide (see Lemma A8.9, \cite{alt1992linear}). Therefore, we also have that 
\begin{align}
\left[\nabla_\Gamma u\cdot\mu\right]^{(2)}_{(1)}=&0&&a.e.\text{ on }\gamma.
\end{align}

\begin{prb}[Sharp interface minimisation problem]
	\label{prob-sharp}
	Find $(u^*,\gamma^*)\in\mathcal{K}_{SI}$ such that
	\begin{align}
	\mathcal{E}_{SI}(u^*,\gamma^*)=\inf_{(u,\gamma)\in\mathcal{K}_{SI}}\mathcal{E}_{SI}(u,\gamma).
	\end{align}
\end{prb}
Minimisers in $\mathcal K_{SI}$ of $\mathcal E_{SI}$ are critical points of the following  Lagrangian $\mathcal L_{SI}$.
\begin{defin} We define the sharp interface Lagrangian by 
	$$\mathcal L_{SI}(u,\gamma,\lambda):=\mathcal E_{SI}(u,\gamma)+\lambda_1C_1(\gamma)+\lambda_2C_2( u)+\sum_{i=1}^3\lambda_{i+2}\mathcal{N}_i(u)$$
	for $\gamma\in C^1$ an embedded curve on $\Gamma$, $u\in H^2(\Gamma)$ and $\lambda\in \mathbb R^5$.
\end{defin}

In the following  two subsections we will define and calculate the first variation of the sharp interface energy and the constraint functionals. Using the function $\chi_{\gamma}$ \eqref{eqn-chi-gamma}, and the geodesic curvature $H_\gamma$ defined by $H_\gamma = h \cdot \mu$ for curvature vector $h$, we derive the following result. 

\begin{pro}
	\label{Prop:SI-min}
	A  pair $(u^*,\gamma^*)\in\mathcal{K}_{SI}$ which minimises the sharp interface energy subject to the constraints \eqref{eqn-constraint1}-\eqref{eqn-constraint3}, and is sufficiently regular so that all the following terms are well defined, solves the following free boundary problem 
	\begin{align}
	\label{eqn-EL-gamma1}
	\begin{split}
	&\left(\Delta_\Gamma+\frac{2}{R^2}\right)\left(\kappa\Delta_\Gamma u^*-\sigma u^*+\kappa\Lambda(\chi_{\gamma^*}-\alpha)\right)=0\end{split}\quad\text{on }\Gamma^{(1)}\cup\Gamma^{(2)},\\
	\label{eqn-EL-gamma}
	\begin{split}
	&\widehat{b}H_{\gamma^*} - \kappa\Lambda\left[\chi_{\gamma^*}\left(\Delta_ \Gamma u^*+\frac{2}{R^2}u^*\right)\right]^{(2)}_{(1)}=
	\\
	&\qquad\qquad\qquad+\dashint_{\gamma^*} \left(\widehat{b}H_{\gamma^*} - \kappa\Lambda\left[\chi_{\gamma^*}\left(\Delta_\Gamma u^*+\frac{2}{R^2}u^*\right)\right]^{(2)}_{(1)}\right)\:{\rm d}\gamma^*
	\end{split}\quad\text{on }\gamma^*,
	\end{align}
	with jump conditions
	\begin{align}
	\label{eqn-jumpconditions2}
	[\nabla_\Gamma \Delta_\Gamma u^*\cdot\mu]^{(2)}_{(1)}&=0,
	&
	[\Delta_\Gamma u^*]^{(2)}_{(1)}&=-2\Lambda \quad & \text{on } \gamma^*. 
	\end{align}
\end{pro}

\begin{rem}
	We note that equations \eqref{eqn-EL-gamma1} and \eqref{eqn-jumpconditions2} are order $\mathcal{O}(\rho)$ approximations and \eqref{eqn-EL-gamma} is an order $\mathcal{O}(\rho^2)$ approximation of the sharp interface equilibrium equations given in Problem 3.10 in \cite{elliott2010surface}. 
	In this case, the Lagrange multiplier $\lambda_A$ for the area constraint is interpreted as the surface tension $\sigma$. 
\end{rem}

\subsection{Variation of the membrane height}

As a first step to prove Proposition \ref{Prop:SI-min} we consider the variation with respect to $u$ in the direction  $\zeta \in H^2(\Gamma)$ whilst keeping $\gamma$ fixed. It is defined in the usual sense and are fairly straightforward to compute. For the constraints \eqref{eqn-constraint2} and \eqref{eqn-constraint3} we obtain that 
\begin{align}
\left<C_2^\prime(u),(\zeta)\right>=&\int_{\Gamma}\zeta\:{\rm d}\Gamma, \\
\left< \mathcal{N}_i^\prime(u),(\zeta)\right>=&\int_{\Gamma}\zeta\nu_i\:{\rm d}\Gamma  ~~\text{ for } i\in\{1,2,3\}, 
\end{align}
whilst for the sharp interface energy the results it 
\begin{multline}
\left<\mathcal{E}^\prime_{SI}(u,\gamma),(\zeta,0)\right> \\
= \int_{\Gamma} \left(\kappa\Delta_\Gamma u\Delta_\Gamma \zeta +\left(\sigma-\frac{2\kappa}{R^2}\right)\nabla_\Gamma u\cdot\nabla_\Gamma \zeta 
-\frac{2\sigma}{R^2} u\zeta +\kappa\Lambda\chi_\gamma\Delta_\Gamma \zeta+\frac{2\kappa\Lambda}{R^2}\chi_\gamma\zeta \right)\:{\rm d}\Gamma. 
\end{multline}

Let $(u^*,\gamma^*,\lambda^*)$ be a critical point of the Lagrangian $\mathcal L_{SI}$ sufficiently smooth such that all terms are well defined for the remainder of this subsection. In this point the first variation of the Lagrangian vanishes, giving the variational equation
\begin{multline}
\label{eqn-SI-Var}
\int_{\Gamma} \left(
\kappa\Delta_\Gamma u^*\Delta_\Gamma \zeta +\left(\sigma-\frac{2\kappa}{R^2}\right)\nabla_\Gamma u^*\cdot\nabla_\Gamma \zeta 
-\frac{2\sigma}{R^2} u^*\zeta +\kappa\Lambda\chi_{\gamma^*}\Delta_\Gamma \zeta+\frac{2\kappa\Lambda}{R^2}\chi_{\gamma^*}\zeta \right)\:{\rm d}\Gamma\\
+\int_\Gamma \left(\lambda_2^*\zeta+\sum_{i=1}^3\lambda_{i+2}^*\nu_i\zeta \right)\:{\rm d}\Gamma =0,\quad\forall\zeta\in H^2(\Gamma).
\end{multline}
By testing with $\zeta=1$ and using that $(u^*,\gamma^*)$ satisfies \eqref{eqn-constraint2} we obtain that 
\begin{align}
\label{eqn-lagrange1(SI)}
\lambda_2^* = \frac{2\kappa\Lambda}{R^2} \frac{|\Gamma^{(1)}|-|\Gamma^{(2)}|}{|\Gamma|} = -\frac{2\kappa\Lambda\alpha}{R^2}.
\end{align} 
By testing with $\zeta=\nu_i$ and using that $(u^*,\gamma^*)$ satisfies \eqref{eqn-constraint3} then we find that  
\begin{align}
\label{eqn-lagrange2(SI)}
\lambda_3^*=\lambda_4^*=\lambda_5^*=0, 
\end{align} 
where we have used that $-\Delta_\Gamma \nu_i=\frac{2}{R^2}\nu_i$ (see \cite{elliott2016small}). Integrating \eqref{eqn-SI-Var} by parts we calculate that
\begin{align}
\begin{split}
0=& \int_{\Gamma^{(1)}} \left(\left(\Delta_\Gamma+\frac{2}{R^2}\right)\left(\kappa\Delta_\Gamma u^*-\sigma u^*+\kappa\Lambda(-1-\alpha)\right)\right)\zeta\:{\rm d}\Gamma^{(1)}  \\
& +\int_{\Gamma^{(2)}} \left(\left(\Delta_\Gamma+\frac{2}{R^2}\right)\left(\kappa\Delta_\Gamma u^*-\sigma u^*+\kappa\Lambda(1-\alpha)\right)\right)\zeta\:{\rm d}\Gamma^{(2)}  \\
& -\int_{\gamma}\kappa\left(\Delta_\Gamma {u^*}^{(2)}-\Delta_\Gamma {u^*}^{(1)}+2\Lambda\right)\nabla_\Gamma\zeta\cdot\mu-\left(\nabla_\Gamma\Delta_\Gamma {u^*}^{(2)}-\nabla_\Gamma\Delta_\Gamma {u^*}^{(1)}\right)\cdot\mu\zeta\:{\rm d}\gamma,
\end{split}
\end{align}
for all $\zeta\in H^2(\Gamma)$. This proves \eqref{eqn-EL-gamma1} and \eqref{eqn-jumpconditions2} from Proposition \ref{Prop:SI-min}.

\subsection{Variation of the interface}	

The second step to prove Proposition \ref{Prop:SI-min} is to calculate the first variation of the Lagrangian $\mathcal L_{SI}$ with respect to $\gamma$. This variation is defined by the instantaneous change of the energy due to the deformation of the interface between rafts and non-rafts regions. 

Given any smooth tangential vector field $v : \Gamma \to \mathbb{R}^3$ let $x(\tau)$ be the solution to $x'(\tau) = v(x(\tau))$ and then 
\begin{align}
 \Gamma^{(i)}(\tau) &= \{ x(\tau) \, | \, x(0) \in \Gamma^{(i)} \}, \quad i=1,2, \\
 \gamma(\tau) &= \{ x(\tau) \, | \, x(0) \in \gamma \}. 
\end{align}
Thanks to the smoothness of $v$, for all $\tau$ close to $0$ an admissible two-phase surface $\Gamma = \Gamma^{(1)}(\tau) \cup \gamma(\tau) \cup \Gamma^{(2)}(\tau)$ is obtained in the sense that $(u,\gamma(\tau)) \in \mathcal{K}_{SI}$. The variation of the Lagrangian is defined as 
\begin{equation}
 \langle \mathcal{L}^\prime_{SI}(u,\gamma,\lambda), (0,v,0) \rangle = \frac{d}{d\tau} \mathcal{L}_{SI}(u,\gamma(\tau),\lambda) \Big{|}_{\tau = 0}. 
\end{equation}
Regarding derivatives of $\tau$ dependent domains we note the following identities that are, for instance, proved in \cite{dziuk2013finite}, Theorem 5.1: For any smooth function $f : \Gamma \to \mathbb{R}$ 
\begin{equation}
 \frac{d}{d\tau} \int_{\Gamma^{(i)}(\tau)} f \:{\rm d}\Gamma^{(i)}(\tau) = \int_{\gamma(\tau)} f v \cdot \nu^{(i)}\:{\rm d}\gamma(\tau), \quad i=1,2, 
\end{equation}
and moreover 
\begin{equation}
 \frac{d}{d\tau} \int_{\gamma(\tau)} 1 \:{\rm d}\gamma(\tau) = \int_{\gamma(\tau)} H_\gamma v \cdot \mu \:{\rm d}\gamma(\tau).
\end{equation}

Using these identities the variation of the sharp interface energy can be calculated, which yields 
\begin{equation} 
\left< \mathcal{E}^\prime_{SI} (u,\gamma),(0,v) \right> = \int_{\gamma} \left[ \hat{b} H_{\gamma} - \kappa \Lambda \left( \Delta_\Gamma u^{(2)} + \Delta_\Gamma u^{(1)} + \frac{4 u}{R^2} \right) \right] v \cdot \mu\:{\rm d}\gamma. 
\end{equation} 
For the constraint functional \eqref{eqn-constraint1} we obtain that 
\begin{equation}
 \left< {C}_1^\prime (\gamma), v \right> = 2 \int_{\gamma} v \cdot \mu \:{\rm d}\gamma. 
\end{equation}
If $(u^*,\gamma^*,\lambda^*)$ is a critical point of the Lagrangian $\mathcal{L}_{SI}$ then 
\begin{align}
0 = \langle \mathcal{L}_{SI}^\prime(u^*,\gamma^*,\lambda^*), (0,v,0) \rangle = \left< \mathcal{E}^\prime_{SI}(u^*,\gamma^*), (0,v) \right> +\lambda_1^* \left< {C}_1^\prime(\gamma^*), v \right>
\end{align}
for all smooth tangential vector fields $v$ on $\Gamma$. We find that
\begin{align}
2\lambda_1^* = \dashint_{\gamma^*} \Big{(} {\widehat{b}H_{\gamma^*}} - \frac{4 \kappa \Lambda u^*}{R^2} - \kappa \Lambda \big{(} \Delta_\Gamma {u^*}^{(1)} + \Delta_\Gamma {u^*}^{(2)} \big{)} \Big{)}\:{\rm d}\gamma^*, 
\end{align}
and hence obtain \eqref{eqn-EL-gamma}. This completes the proof of Proposition \ref{Prop:SI-min}.

\subsection{Reduced sharp interface energy}
Analogous to Subsection \ref{Reduced diffuse interface energy} we introduce a reduced sharp interface energy. By using the Euler-Lagrange equation \eqref{eqn-EL-gamma1}, that $BV(\Gamma)\hookrightarrow L^2(\Gamma)$  (Corollary 3.49, \cite{ambrosio2000functions}) and repeating the argument of  Subsection \ref{Reduced diffuse interface energy} we find that
\begin{align}
\hf^*=\mathcal{G}(\mathbf{P}(\chi_{\gamma^*})).
\label{eqn-Defhf-2}
\end{align}
Here, $\mathcal{G}$ is the Green's function defined in \eqref{eqn-green} and $\mathbf{P}$ is the $L^2-$projection onto \\$\text{span}\{1,\nu_1,\nu_2,\nu_3\}^\perp$.

By elliptic regularity we have that $\mathcal{G}(\mathbf{P}(\chi_{\gamma^*}))\in H^2(\Gamma)$, and hence that $(\gamma^*,\mathcal{G}(\mathbf{P}(\chi_{\gamma^*}))) \in\mathcal{K}_{SI}$. 

Using \eqref{eqn-Defhf-2} we define the \emph{reduced sharp interface energy}, 
\begin{align}
\label{eqn-reduced-sharp}
\widetilde{\mathcal{E}}_{SI}(\gamma^*):=\mathcal{E}_{SI}(\mathcal{G}(\mathbf{P}(\chi_{\gamma^*})),\gamma^*).
\end{align}
Using the same method as Subsection \ref{Sec:Gamma-conv} we simplify the reduced sharp interface energy to obtain that
\begin{align}
\label{eqn-refucedenergy-2}
\widetilde{\mathcal{E}}_{SI}(\gamma^*)=\begin{cases}
\int_{\Gamma}\frac{\kappa\Lambda}{2}\mathbf{P}(\chi_{\gamma^*})\left(\frac{\sigma}{\kappa}+\frac{2}{R^2}\right)\mathcal{G}(\mathbf{P}(\chi_{\gamma^*}))+\frac{\kappa\Lambda^2}{2}(\chi_{\gamma^*}^2-(\mathbf{P}(\chi_{\gamma^*}))^2)\:{\rm d}\Gamma+\int_{\gamma^*}\widehat{b}\:{\rm d}\gamma&\text{for }\gamma^*\in\widetilde{\mathcal{K}}_{SI},\\
+\infty &\text{otherwise,}
\end{cases}
\end{align}
where 
\begin{align}
\widetilde{\mathcal{K}}_{SI}:=&\left\{\gamma^*\in C^1(\Gamma) : |\Gamma^{(1)}|-|\Gamma^{(2)}|+\alpha|\Gamma|=0 \right\}.
\end{align}
and $\widehat{b}=c_Wb$, see Proposition \ref{Prop-GammaLimit}.

We highlight that the reduced sharp interface energy coincides with the $\Gamma-$limit of the diffuse interface energy calculated in Subsection \ref{Sec:Gamma-conv}.

We write the constrained minimisation problem for the reduced sharp interface energy below.
\begin{prb}
	\label{prob-sharp-reduced}
	Find $\widetilde{\gamma^*}\in \widetilde{\mathcal{K}}_{SI}$ such that
	\begin{align}
	\widetilde{\mathcal{E}}_{SI}(\widetilde{\gamma^*})=\inf_{\gamma \in \widetilde{\mathcal{K}}_{SI}}\widetilde{\mathcal{E}}_{SI}(\gamma).
	\end{align}	
\end{prb}
Finally, we note that finding a minimiser of  Problem \ref{prob-sharp-reduced} is equivalent to finding a  minimiser to Problem \ref{prob-sharp} since
\begin{align}
\label{eqn:red-ord-3-2}
\mathcal{E}_{SI}(\gamma^*,u^*) \leq \mathcal{E}_{SI}(\widetilde{\gamma}^*,\mathcal{G}(\mathbf{P}(\chi_{\widetilde{\gamma}^*}))) = \widetilde{\mathcal{E}}_{SI}(\widetilde{\gamma}^*) \leq \widetilde{\mathcal{E}}_{SI}(\gamma^*) = \mathcal{E}_{SI}(\gamma^*,\mathcal{G}(\mathbf{P}(\chi_\gamma^*))) = \mathcal{E}_{SI}(\gamma^*,u^*).
\end{align}

To summarise the previous two sections, we have related the diffuse interface energy \eqref{eqn-diffuse-energy} to the sharp interface energy \eqref{eqn-sharp-interface-energy} as follows.
\begin{enumerate}
	\item Minimisers of the diffuse interface energy \eqref{eqn-diffuse-energy} coincide with minimisers of the reduced diffuse interface energy \eqref{eqn-refucedenergy}.
	\item The reduced diffuse interface energy \eqref{eqn-refucedenergy} $\Gamma-$converges to the reduced sharp interface energy \eqref{eqn-reduced-sharp}.
	\item Minimisers of the reduced sharp interface energy \eqref{eqn-reduced-sharp} coincide with minimisers of the sharp interface energy \eqref{eqn-sharp-interface-energy}.
\end{enumerate}

%%%%%%%%%%%%%%%%%%%%%%%%%%%%%%%%%%%%%%%%%%%%%%%%%%%%%%%%%%%%%%%%%%%%%%%%%%%%%%%%%%%%%
\section{Formal asymptotics for a phase field gradient flow}
\label{Sec-Formal}
We now consider the following time evolution problem,
	\begin{align}
	\label{eqn-Diffuse-2}
	0=&\left(\Delta_\Gamma+\frac{2}{R^2}\right)(\kappa\Delta_\Gamma u-\sigma u +\kappa\Lambda(\phi-\alpha))&&\text{on }\Gamma,
	\\
	\label{eqn-Diffuse-1}
	\beta\epsilon\phi_t=&b\epsilon\Delta_\Gamma\phi-\frac{b}{\epsilon}W^\prime(\phi)-\kappa\Lambda\Delta_\Gamma u - \frac{2\kappa\Lambda u}{R^2}-{\kappa\Lambda^2 \phi}+\lambda&&\text{on }\Gamma,
	\end{align}
	with initial conditions $\phi(0)=\phi_0$ and $u(0)=u_0$ and satisfying the constraints \eqref{eqn-constraint-DI}. Here, $\lambda$  is the Lagrange multiplier associated with the constraint $\dashint_\Gamma \phi=\alpha$ and $\beta>0$ is a kinetic coefficient. These equations were introduced in \cite{elliott_hatcher_2020} as a conserved $L^2-$gradient flow of the diffuse interface energy \eqref{eqn-diffuse-energy}. There they were used to numerically obtain local equilibria of \eqref{eqn-diffuse-energy}, with respect to the phase field whilst assuming that the membrane height is in mechanical equilibrium. These were used to numerically compute local equilibria of \eqref{eqn-diffuse-energy}, which are solutions to the Euler-Lagrange equations \eqref{eqn-diffuse-EL2} and \eqref{eqn-diffuse-EL}. 
	\\
	\\
	Similarly, turning to consider the sharp interface energy \eqref{eqn-sharp-interface-energy}, the following evolution problem can be obtained as a conserved $L^2-$gradient flow for two-phase surfaces $(u(t),\gamma(t))\in\mathcal{K}_{SI}$, \cite{hatcher2020},
\begin{align}
\label{eqn-GF-gamma1}
0=&\left(\Delta_\Gamma+\frac{2}{R^2}\right)\left(\kappa\Delta_\Gamma u-\sigma u+\kappa\Lambda(-1-\alpha)\right)&&\text{on }\Gamma^{(1)}(t),
\\
\label{eqn-GF-gamma2}
0=&\left(\Delta_\Gamma+\frac{2}{R^2}\right)\left(\kappa\Delta_\Gamma u-\sigma u+\kappa\Lambda(1-\alpha)\right)&&\text{on }\Gamma^{(2)}(t),
\\
\label{eqn-GF-gamma}
\begin{split}
\widehat{\beta}\mathcal{V}=&-\widehat{b}H_\gamma+\widehat{b}\dashint_\gamma H_\gamma\:{\rm d}\gamma+\frac{4\kappa\Lambda}{R^2}\left(u-\dashint_\gamma u\:{\rm d}\gamma\right)	
\\
&+\kappa\Lambda\left(\Delta_\Gamma u^{(1)}+\Delta_\Gamma u^{(2)}-\dashint_{\gamma} \left(\Delta_\Gamma u^{(1)}+\Delta_\Gamma u^{(2)}\right)\:{\rm d}\gamma\right)
\end{split}&&\text{on }\gamma(t),
\end{align}
with initial conditions $u(0)=u_0$ and $\gamma(0)=\gamma_0$ satisfying the constraints \eqref{eqn-constraint1}-\eqref{eqn-constraint3}, and jump conditions across $\gamma(t)$,
\begin{align}
\label{eqn-GF-jumpconditions1}
[u]^{(2)}_{(1)}&=0,
&
[\nabla_\Gamma u\cdot\mu]^{(2)}_{(1)}&=0
\\
\label{eqn-GF-jumpconditions2}
[\Delta_\Gamma u]^{(2)}_{(1)}&=-2\Lambda,
&
[\nabla_\Gamma \Delta_\Gamma u\cdot\mu]^{(2)}_{(1)}&=0.
\end{align}
Here $\mathcal{V}(t)$ is the velocity of $\gamma(t)$ in the direction of the co-normal $\mu$ and $\widehat{\beta}=c_W\beta$. Note that stationary solutions of the gradient flow equations  \eqref{eqn-GF-gamma1}-\eqref{eqn-GF-jumpconditions2} are the Euler-Lagrange equations \eqref{eqn-EL-gamma1}-\eqref{eqn-jumpconditions2}.

Our objective in this section is to show that the limiting problem of \eqref{eqn-Diffuse-2}--\eqref{eqn-Diffuse-1} as $\epsilon \to 0$ is \eqref{eqn-GF-gamma1}--\eqref{eqn-GF-jumpconditions2}. For this purpose we note that since \eqref{eqn-Diffuse-2} coincides with \eqref{eqn-diffuse-EL}, then the calculation given in Section \ref{Reduced diffuse interface energy} can be repeated here and \eqref{eqn-Diffuse-2} can be reduced to \eqref{eqn-Defhf}. This has the benefit of only considering a second order equation instead of a fourth order equation for the height function, but has the added cost of involving the non-local projection operator $\mathbf{P}$. To deal with the non-local term it will prove helpful to write $p:=\mathbf{P}(\phi)$. 

 Hence, the system \eqref{eqn-Diffuse-2}-\eqref{eqn-Diffuse-1} can be reduced to
 \begin{align}
 \label{Diffuse2}
 p=&-\Delta_\Gamma u+\frac{\sigma}{\kappa}u&&\text{on }\Gamma,
 \\
 \label{Diffuse2b}
 p=&\mathbf{P}(\phi)&&\text{on }\Gamma,
 \\
 \label{Diffuse1}
 \beta \epsilon\phi_t=&b\epsilon\Delta_\Gamma\phi-\frac{b}{\epsilon}W^\prime(\phi)-\kappa\Lambda\Delta_\Gamma u - \frac{2\kappa\Lambda u}{R^2}-{\kappa\Lambda^2 \phi}+\lambda&&\text{on }\Gamma.
 \end{align}
 It is to \eqref{Diffuse2}-\eqref{Diffuse1} we perform a formal asymptotic analysis based on matching asymptotic $\epsilon$ expansions in the diffuse interfaces and in the bulk phases away from the interfaces. The technique is well established for phase field models, for instance, see \cite{FP95} for details of the procedure. We denote by $(\phi_\epsilon,u_\epsilon,\lambda_\epsilon, p_\epsilon)$ a family of solutions to \eqref{Diffuse2}-\eqref{Diffuse1} that converges formally to some limit denoted by $(\phi,u,\lambda,p)$. We assume that $\phi=\chi_\gamma$ for some smooth curve $\gamma$ that separates the regions $\Gamma^{(1)}=\{(x,t)\in \Gamma\times [0,T] :\phi(x,t)=-1 \}$ and $\Gamma^{(2)}=\{(x,t)\in \Gamma\times [0,T] :\phi(x,t)=+1\}$, see \eqref{eqn-chi-gamma}. Using that $\mathbf{P}$ is the $L^2-$projection onto $\text{span}\{1,\nu_1,\nu_2,\nu_3\}^\perp$ we calculate that
 \begin{align}
 \label{eqn-projection}
 \mathbf{P}(\phi)=\begin{cases}
 -1-\frac{1}{|\Gamma|}\left(|\Gamma^{(2)}|-|\Gamma^{(1)}|\right)-\frac{1}{3|\Gamma|}\sum_{i=1}^{3}\nu_i\left(\int_{\Gamma^{(2)}}\nu_i-\int_{\Gamma^{(1)}}\nu_i\right)&\text{on }\Gamma^{(1)}\\
 +1-\frac{1}{|\Gamma|}\left(|\Gamma^{(2)}|-|\Gamma^{(1)}|\right)-\frac{1}{3|\Gamma|}\sum_{i=1}^{3}\nu_i\left(\int_{\Gamma^{(2)}}\nu_i-\int_{\Gamma^{(1)}}\nu_i\right)&\text{on }\Gamma^{(2)}.
 \end{cases}
 \end{align}
 
 We will show that the limit solution $(\phi,u,\lambda,p)$ satisfies the following free boundary value problem on $\Gamma$,
\begin{align}
\label{Sharp1}
\begin{split}
\phi=&- 1\\
\phi=&+ 1
\end{split}&
\begin{split}
&\text{on }\textbf{$\Gamma$}^{(1)}\\
&\text{on }\textbf{$\Gamma$}^{(2)}
\end{split}
\\
\label{Sharp2}
\begin{split}
-\Delta_\Gamma u+\frac{\sigma}{\kappa}u=&\Lambda p
\end{split}&
\begin{split}
&\text{on }\textbf{$\Gamma$}^{(1)}\cup\textbf{$\Gamma$}^{(2)}\qquad
\end{split}
\\
\label{Sharp3}
\begin{split}
p=&\mathbf{P}(\phi)
\end{split}&
\begin{split}
&\text{on }\textbf{$\Gamma$}^{(1)}\cup\textbf{$\Gamma$}^{(2)}
\end{split}
\\
\label{Sharp5}
[u]^{(2)}_{(1)}=&0&&\text{on }\gamma
\\
\label{Sharp6}
[\nabla_\Gamma u]^{(2)}_{(1)}\cdot\mu=&0&&\text{on }\gamma
\\
\label{Sharp9}
\begin{split}
\widehat{\beta}\mathcal{V}=&-\widehat{b}H_\gamma+\widehat{b}\dashint_\gamma H_\gamma\:{\rm d}\gamma+\left(\frac{4\kappa\Lambda}{R^2}+2\sigma\Lambda\right)\left(u-\dashint_\gamma u\:{\rm d}\gamma\right)
\\
&-\kappa\Lambda^2\left(\mathbf{P}(\phi)^{(1)}+\mathbf{P}(\phi)^{(2)}-\dashint_\gamma \mathbf{P}(\phi)^{(1)}+\mathbf{P}(\phi)^{(2)}\:{\rm d}\gamma\right)
\end{split}&&\text{on }\gamma
\end{align}
We comment that \eqref{eqn-GF-gamma1}-\eqref{eqn-GF-jumpconditions2} can be obtained from \eqref{Sharp1}-\eqref{Sharp9}. Firstly, by combining \eqref{Sharp2} and \eqref{Sharp3}, applying the operator $\left(\Delta_\Gamma+\frac{2}{R^2}\right)$, and using that the components of the normal $\nu_i$ are in the null space of $\left(\Delta_\Gamma+\frac{2}{R^2}\right)$ we obtain \eqref{eqn-GF-gamma1} and \eqref{eqn-GF-gamma2}. Secondly, again combining \eqref{Sharp2} and \eqref{Sharp3}, and substituting this into \eqref{Sharp9} to eliminate $\mathbf{P}(\phi)^{(i)}$ we obtain \eqref{eqn-GF-gamma}. Finally using \eqref{eqn-projection} we calculate that $[\mathbf{P}(\phi)]^{(2)}_{(1)}=2$ and $[\nabla_\Gamma\mathbf{P}(\phi)]^{(2)}_{(1)}\cdot\mu=0$. Hence, we obtain \eqref{eqn-GF-jumpconditions1} and \eqref{eqn-GF-jumpconditions2} from \eqref{Sharp2}, \eqref{Sharp5} and \eqref{Sharp6}.
Altogether, we see that the equations \eqref{Sharp1}-\eqref{Sharp9} indeed yield the sharp-interface gradient flow equations \eqref{eqn-GF-gamma1}-\eqref{eqn-GF-gamma}. Therefore it only remains to show \eqref{Sharp1}-\eqref{Sharp9}.    
\subsection{Matching conditions}
As is standard for these problems we will consider outer expansions (solutions that are only valid away from the interface $\gamma$) and inner expansions (solutions that are only valid near to the interface). We consider inner expansions in addition to the outer expansions since near the interface $\gamma$ it's possible that there could be very steep transition layers. Therefore the derivatives could contribute non-zero order $\mathcal{O}(\epsilon)$ terms which need to be accounted for. On the region where both inner and outer expansions are valid matching conditions relate the outer expansions to the inner expansions.
\\
\\	
	We will assume that the outer expansions have the form
	\begin{align}
	f_\epsilon(x,t)=\sum_{k=0}^{\infty}\epsilon^kf_k(x,t)
	\end{align}
	where $f_\epsilon=\phi_\epsilon,u_\epsilon, \lambda_\epsilon$ or $p_\epsilon$. To write down the inner expansions we consider a parameterisation $\Theta(s,r,t)$ such that $s\mapsto\Theta(s,0,t)$ gives a parameterisation of $\gamma(t)$ and $r$ denotes the signed geodesic distance of a point $x=\Theta(s,r,t)\in \Gamma$ to the interface $\gamma(t)$. Further details of a suitable parameterisation for the sphere can be found in \cite{garcke2016coupled}. Since the length scale of the transition layers is $\epsilon$ we introduce the parameter $z$ given by
	\begin{align}
	z=\frac{r}{\epsilon}.
	\end{align}
	We then assume that the inner expansions are of the form
	\begin{align}
	f_\epsilon(x,t)=F(s,z,t;\epsilon)=\sum_{k=0}^{\infty}\epsilon^k F_k(s,z,t)
	\end{align}
	where again $f_\epsilon=\phi_\epsilon,u_\epsilon,\lambda_\epsilon$ or $p_\epsilon$ with $F_k=\Phi_k,U_k,L_k$ or $P_k$ respectively. On the region where both outer and inner expansions are valid we prescribe the following matching conditions to ensure consistency, 
\begin{align}
\label{Match1}
F_0(s,\pm\infty,t)&\sim f^\pm_0(x,t)\\
\label{Match2}
\partial_z F_0(s,\pm \infty, t)&\sim 0
\\
\label{Match3}
\partial_z F_1(s,\pm \infty, t)&\sim\nabla_\Gamma f^\pm_0(x,t)\cdot\mu(x,t),
\end{align}
where $f_0^\pm(x,t)=\lim_{\delta\to 0}f(\Theta(s,\pm\delta,t),t)$. A derivation of these matching conditions can be found in \cite{garcke2006second}.
\subsection{Outer expansions}
We begin by matching orders of $\epsilon$ for the outer expansions first. In \cite{rubinstein1992nonlocal} Rubinstein and Sternberg considered a formal asymptotic analysis for the conserved Allen-Cahn equation and demonstrated that for faster timescales it is sufficient to suppose the lowest order term of the Lagrange multiplier is of order $\mathcal{O}(\epsilon^0)$. Their analysis can equally we applied to our system of equations and so we make the assumption that the lowest order term of the Lagrange multiplier $\lambda$ is of order $\mathcal{O}(\epsilon^0)$.  Hence, considering terms of order $\mathcal{O}(\epsilon^{-1})$ in \eqref{Diffuse1}, we obtain that 
\begin{equation}
W^\prime(\phi_0)=0
\end{equation}
and hence the only stable solutions are
\begin{align}
\label{eqn-outer-phi}
\phi_0=\pm 1.
\end{align} 
 Therefore we deduce that $\phi_\epsilon\to \pm 1$, which justifies \eqref{Sharp1}. Furthermore by considering terms of order $\mathcal{O}(\epsilon^0)$ in \eqref{Diffuse2} and \eqref{Diffuse2b} we readily obtain \eqref{Sharp2} and \eqref{Sharp3}.

\subsection{Inner expansions}
Before considering the inner expansions we first have to write the Laplace-Beltrami operator and the time derivative in local coordinates near to the interface. Calculations in \cite{garcke2016coupled} show that the following expressions are obtained for a function $f(s,z,t)$ defined on a neighbourhood around the interface $\gamma$
\begin{align}
\label{coords1}
\Delta_\Gamma f=\frac{1}{\epsilon^2}\partial_{zz}f+\frac{H_{\gamma}}{\epsilon}\partial_z f+\Delta_{\gamma}f+\mathcal{O}(\epsilon)
\end{align}
and
\begin{align}
\label{coords2}
\frac{d}{dt} f=-\frac{1}{\epsilon}\mathcal{V}\partial_z f+\partial_t f+\mathcal{O}(\epsilon).
\end{align}
where $\Delta_\gamma$ is the Laplace-Beltrami operator along the curve $\gamma$.
\\
\\
Since we have assumed that the limit as $\epsilon\to 0$ exists, it follows that the terms of leading order in $\epsilon$ cancel out. We denote the inner expansions of $\phi_\epsilon,u_\epsilon,\lambda_\epsilon$ by $\Phi,U$ and $L$ respectively. We begin by considering terms of order $\mathcal{O}(\epsilon^{-2})$ and $\mathcal{O}(\epsilon^{-1})$ in \eqref{Diffuse2} to obtain that
\begin{align}
\label{3o2}
\partial_{zz}U_0&=0\\
\label{3o1}
H_\gamma\partial_z U_0+\partial_{zz} U_1&=0
\end{align}
Integrating \eqref{3o2} from $-\infty$ to $z$ and using the matching condition \eqref{Match2} we obtain that
\begin{equation}
\partial_z U_0=0
\label{H-independant-z}
\end{equation}
Hence
\begin{equation}
U_0(z=+\infty)=U_0(z=-\infty)
\end{equation}
from which we obtain \eqref{Sharp5} with the matching condition (\ref{Match1}). 
\\
\\
Similarly integrating \eqref{3o1} from $-\infty$ to $\infty$, and using the matching condition \eqref{Match3} we obtain \eqref{Sharp6}.

The terms of order $\mathcal{O}(\epsilon^{-1})$ in \eqref{Diffuse1} are 
\begin{align}
0=b\partial_{zz}\Phi_0-bW^\prime(\Phi_0)-\kappa\Lambda(H_\gamma\partial_z U_0+\partial _{zz}U_1)
\end{align}
which using \eqref{3o1} simplifies to
\begin{equation}
\label{1o1}
0=b\partial_{zz}\Phi_0-bW^\prime(\Phi_0).
\end{equation} 
Hence, using the outer expansion \eqref{eqn-outer-phi} and matching condition \eqref{Match1} we have that $\Phi_0(z,s,t)$ is a solution of 
\begin{align}
\label{eqn-Phi0}
\begin{cases}
\partial_{zz}\Phi_0=W^\prime(\Phi_0)\\
\Phi_0(\pm\infty)=\pm 1.
\end{cases}
\end{align}
%as $t\to\infty$, then 
We find that $\Phi_0(z,s,t)$ is independent of $s$ and $t$ and given by
\begin{equation}
\label{tanh}
\Phi_0(z) = \tanh\left(\frac{z}{\sqrt{2}}\right).
\end{equation}
Finally, we consider terms of order $\mathcal{O}(\epsilon^0)$ in \eqref{Diffuse1}, 
\begin{align}
\begin{split}
-\beta\partial_z\Phi_0\mathcal{V}=& \, bH_\gamma\partial_z\Phi_0-bW^{\prime\prime}(\Phi_0)\Phi_1+L_0+b\partial_{zz}\Phi_1-{\kappa\Lambda^2\Phi_0}
\\
&-\kappa\Lambda\left(\Delta_\gamma U_0+H_\gamma\partial_z U_1+\partial_{zz}U_2\right)-\frac{2\kappa\Lambda U_0}{R^2}.
\label{1o0}
\end{split} 
\end{align}
Considering the terms of order $\mathcal{O}(\epsilon^0)$ in \eqref{Diffuse2} we obtain 
\begin{equation}
\label{3o0}
-\left(\Delta_\gamma U_0+H_\gamma\partial_z U_1+\partial_{zz} U_2\right)=\Lambda P_0-\frac{\sigma}{\kappa}U_0. 
\end{equation}
Using \eqref{3o0} to simplify \eqref{1o0} we obtain that 
	\begin{align}
	\begin{split}
-\beta\partial_z\Phi_0\mathcal{V} =& bH_\gamma\partial_z\Phi_0-bW^{\prime\prime}(\Phi_0)\Phi_1+L_0+b\partial_{zz}\Phi_1 \\
&-\kappa\Lambda^2\Phi_0+\kappa\Lambda\left(\Lambda P_0-\left(\frac{\sigma}{\kappa}+\frac{2}{R^2}\right)U_0\right). 
\end{split}
\end{align}
It is straightforward to show that the function $\partial_z \Phi_0$ is in the kernel of the operator $-\partial_{zz} + W''(\Phi_0)$. To ensure solvability of the equation for $\Phi_1$ the source term has to be orthogonal to $\partial_z \Phi_0$ with respect to the $L^2$ inner product. We refer to \cite{alfaro2008singular}, Lemma 2.2 for the details (see also Lemma 4.1 in \cite{garcke2016coupled}). This condition reads 
\begin{multline}
0=\int_{-\infty}^{+\infty}-\beta(\partial_z\Phi_0)^2\mathcal{V}-bH_\gamma(\partial_z\Phi_0)^2-{L_0}\partial_z\Phi_0\\-\kappa\Lambda\left(\Lambda P_0-\left(\frac{\sigma}{\kappa}+\frac{2}{R^2}\right)U_0\right)\partial_z\Phi_0+\frac{\kappa\Lambda^2}{2}\partial_z(\Phi_0^2)\:{\rm d}z
\end{multline}
Applying \eqref{eqn-Phi0} we see that the last term vanishes. Using \eqref{tanh} we calculate that $\int_{-\infty}^{\infty}(\partial_z\Phi_0)^2{\rm d}z=\frac{2\sqrt{2}}{3}$, and using that  $\widehat{b}=\frac{2\sqrt{2}b}{3}$ and  $\widehat{\beta}=\frac{2\sqrt{2}\beta}{3}$ it follows that
	\begin{equation}
\widehat{\beta}\mathcal{V}=-\widehat{b}H_\gamma-{2L_0}-\kappa\Lambda\int_{-\infty}^{+\infty}\left(\Lambda P_0-\left(\frac{\sigma}{\kappa}+\frac{2}{R^2}\right)U_0\right)\partial_z\Phi_0\:{\rm d}z
\end{equation}
where above we have used that $\Phi_0(\pm \infty)=\pm 1$. 

We recall from Subsection \ref{Reduced diffuse interface energy} that $\mathbf{P}$ is the $L^2$-projection onto $\text{span}\{1,\nu_1,\nu_2,\nu_3\}^\perp$. Then since $\dashint_\Gamma\phi\:{\rm d}\Gamma=\alpha$ and the components of the normal, $\nu_i$, are in the kernel of the operator $\left(\Delta_\Gamma+\frac{2}{R^2}\right)$ it follows that  
\begin{align}
\label{eqn-proj}
\left(\Delta_\Gamma+\frac{2}{R^2}\right)(\phi-\alpha)=\left(\Delta_\Gamma+\frac{2}{R^2}\right)\mathbf{P}(\phi).
\end{align}
Considering terms of $\mathcal{O}(\epsilon^{-2})$ in \eqref{eqn-proj} we obtain that
\begin{align}
\partial_{zz}\Phi_0=\partial_{zz}P_0.
\end{align}
So by integrating this and using the matching condition \eqref{Match2} it follows that
\begin{align}
\label{eqn-inner-projection}
\partial_z\Phi_0=\partial_z P_0.
\end{align}
Using \eqref{H-independant-z} (that is $U_0$ is independent of $z$) and \eqref{eqn-inner-projection} we obtain that
\begin{equation}
\widehat{\beta}\mathcal{V}=-\widehat{b}H_\gamma-{2L_0}-\frac{\kappa\Lambda^2}{2}\int_{-\infty}^{\infty}\partial_z((P_0)^2)\:{\rm d}z+{\kappa\Lambda}\left(\frac{\sigma}{\kappa}+\frac{2}{R^2}\right)U_0\int_{-\infty}^{\infty}\partial_z\Phi_0\:{\rm d}z
\end{equation}
which using \eqref{eqn-Phi0} simplifies to give
\begin{align}
\widehat{\beta}\mathcal{V}&=-\widehat{b}H_\gamma-{2L_0}+{2\kappa\Lambda}\left(\frac{\sigma}{\kappa}+\frac{2}{R^2}\right)U_0-\frac{\kappa\Lambda^2}{2}\left((P_0(+\infty))^2-(P_0(-\infty))^2\right).
\end{align}
By integrating \eqref{eqn-inner-projection} and using the matching condition \eqref{Match1} we obtain that 
\begin{align}
\label{eqn-inner-projection2}
P_0(+\infty)-P_0(-\infty)=\Phi_0(+\infty)-\Phi_0(-\infty)=2.
\end{align}
Therefore, using \eqref{eqn-inner-projection2} gives that
\begin{align}
\label{preSharp9}
\begin{split}
\widehat{\beta}\mathcal{V}=&-\widehat{b}H_\gamma-{2L_0}+\left(\frac{4\kappa\Lambda}{R^2}+2\sigma\Lambda\right)U_0-{\kappa\Lambda^2}\left(P_0(-\infty)+ P_0(+\infty)\right).
\end{split}
\end{align}
It remains to determine $L_0$, for which we use the constraint $\dashint_{\Gamma}\phi_\epsilon=\alpha$ (a similar example can be found in \cite{berlyand2017sharp}). Hence it follows using \eqref{coords2} and considering terms of order $\mathcal{O}(\epsilon^{-1})$ that
\begin{align}
0=\int_{\gamma}\mathcal{V}\partial_z\Phi_0\:{\rm d}\gamma
\end{align}
and hence using that $\partial_z\Phi_0$ is independent of $s$ we obtain that
\begin{align}
\int_{\gamma}\mathcal{V}\:{\rm d}\gamma=0.
\end{align}
Integrating \eqref{preSharp9} and using the above result we obtain that
\begin{align}
2L_0=&\dashint_\gamma-\widehat{b}H_\gamma+\left(\frac{4\kappa\Lambda}{R^2}+2\sigma\Lambda\right)U_0-{\kappa\Lambda^2}\left(P_0(-\infty)+ P_0(+\infty)\right)\:{\rm d}\gamma.
\end{align}
Finally applying the matching condition \eqref{Match1} and using \eqref{Sharp3}
gives \eqref{Sharp9}.

%%%%%%%%%%%%%%%%%%%%%%%%%%%%%%%%%%%%%%%%%%%%%%%%%%%%%%%%%%%%%%%%%%%%%%%%%%%%%%%%%%%%%
\section{Conclusion}
\label{Sec-Conclusion}
We have analysed and related  sharp and diffuse interface energies obtained by applying a perturbation approach for two-phase approximately spherical biomembranes. We simplified the diffuse interface energy by using the Euler-Lagrange equations to eliminate the height function in order to obtain what we've referred to as the reduced diffuse interface energy. In particular we showed that the minimisation problem for the original energy is equivalent to the minimisation problem for the reduced energy. Furthermore, we calculated the $\Gamma-$limit of the reduced diffuse interface energy and considered the minimisation problem. This is important since results relating to $\Gamma-$convergence can be used to show that minimisers of \eqref{eqn-diffuse-energy} converge to a minimiser of \eqref{eqn-sharp-interface-energy}.

We then performed a formal asymptotic analysis for a system of gradient flow equations of the diffuse interface energy that had previously been considered in \cite{elliott_hatcher_2020}.  The free boundary problem attained from this analysis coincided with the corresponding gradient flow equations for the sharp interface energy. Here, we again showed how using this reduced energy could simplify this calculation.

\subsection*{Acknowledgements}
The work of CME was partially supported by the Royal Society via a Wolfson Research Merit Award. 
The research of LH was funded by the Engineering and Physical Sciences Research Council grant EP/H023364/1
under the MASDOC centre for doctoral training at the University of Warwick.

\bibliographystyle{acm}
\bibliography{bib}
\end{document}